\documentclass{amsart}

\usepackage{mathptmx}      
\usepackage{mymacros}
\usepackage{color}
\usepackage[mathscr]{eucal}
\usepackage{enumerate}
\usepackage[applemac]{inputenc}
\usepackage[english]{babel}
\usepackage{float}
\usepackage{graphicx}

\definecolor{NavyBlue}{rgb}{0.,0.44,0.7}

\makeatletter
\newcommand{\tpitchfork}{%
  \vbox{
    \baselineskip\z@skip
    \lineskip-.52ex
    \lineskiplimit\maxdimen
    \m@th
    \ialign{##\crcr\hidewidth\smash{$-$}\hidewidth\crcr$\pitchfork$\crcr}
  }%
}
\makeatother

\usepackage[sort&compress,numbers]{natbib}

\begin{document}

\title{Hierarchical stratification of Pareto sets
}




\author{Alberto Lovison }
\address{Dipartimento di Matematica -- Universit\`a degli Studi di Padova\\
              via Trieste, 63 -- 35121 Padova ITALY\\
              Tel.: +39-049-8271310\\
              Fax: +39-049-8271499}
 \email{lovison@math.unipd.it} 

\author{Filippo Pecci}
\address{Dipartimento di Matematica -- Universit\`a degli Studi di Padova\\
              via Trieste, 63 -- 35121 Padova ITALY}
\email{filippo.pecci@studenti.unipd.it}



\maketitle

\begin{abstract} In smooth and convex multiobjective optimization problems the set of Pareto optima  is diffeomorphic to an $m-1$ dimensional simplex, where $m$ is the number of objective functions. 
The vertices of the simplex are the optima of the individual functions and the $(k-1)$-dimensional facets are the Pareto optimal set of $k$ functions subproblems.
Such a hierarchy of submanifolds is a geometrical object called \emph{stratification} and the union of such manifolds, in this case the set of Pareto optima, is called a \emph{stratified set}. 
We discuss how these geometrical structures generalize in the  non convex cases, we survey the known results and deduce possible suggestions for the design of dedicated optimization strategies.\\
%
\end{abstract}


\section{Introduction}

Multiobjective optimization (MO) deals with the problem of optimizing several functions at once. 
\begin{definition}
 	Let $f_1,\dots,f_m:W \to \R$, $W\subseteq \R^n$, $x,y\in W$. $x$ \emph{dominates} $y$ if $f_i(x)\geq f_i(y)$, for all $i$ and $f_j(x)> f_j(y)$ for at least one $j$. If there does not exist any point $y\in W$ dominating $x$, then $x$ is called a \emph{Pareto optimum}. 
\end{definition}
The first important difference with single function optimization is that the set of optima is composed by a unique point only in degenerate cases: usually, even with only two functions, the set of optima consists in an uncountable set of points. The set of all such generalized optima is called the set of Pareto optima. If there exists a neighborhood of $p$ in $W$ where $p$ is Pareto optimal, then $p$ is called a \emph{local Pareto optimum}. The set of local Pareto optima is denoted by $\theta_{op}$.

In most practical cases the goal of MO is to produce a unique solution coping with the several requirements represented by the objective functions, however, the choice of the preferred solution is always the outcome of a compromise involving non universal criteria.  

Being in principle all the possible optimal solutions equivalent, it seems worth to obtain the most precise and complete information on them before passing to the decision stage.\footnote{Among many others, we mention some continuation methods aiming at approximating the whole structure of the Pareto set \cite{Hillermeier:2001fk,Das:1998zh,Delgado-Pineda:2013qf,Miglierina:2008yq,Schutze:2005fk,Martin:2013gf,Custodio:2011rm}.}
Such a complete representation of the Pareto optimal set assumes a special meaning in the point of view of global optimization. 

Furthermore, when a parametric optimization problem can be defined as, e.g., a weighted sum or a more general non linear combination of a number of functions, the optimal solution varies as the parameters vary, spanning a portion of the Pareto set of the MO problem defined by the family of functions. These problems are of particular interest, because a single optimal design should not be selected among the others, but in some sense the whole family of solutions has to be implemented. We will discuss an example in biology of such applications in section \ref{sec:evolution}, but more examples has emerged also recently in completely different research areas, e.g., materials science \cite{Lejaeghere:2013kx,Truskinovsky:1996hf}.

Our goal is to describe the geometrical organization of the set of Pareto optima one is expected to find in the generic smooth but nonlinear and non convex MO problem. 
We will see that in the general case of an MO problem with $m$ objective functions, the set of local Pareto optima $\theta_{op}$ will consist in a \emph{manifold with boundary and corners} of dimension $m-1$, i.e., a geometrical object called a \emph{stratified set} in the sense of Mather \cite{Mather:1973fk,Mather:2012gp}. 
As will be described in the examples of section \ref{sec:evolution}, the components (\emph{strata}) of dimension $k-1$ in the Pareto set are related to the subproblems of $k$ objective functions. Such strata play a special role in the optimization process and are fundamental ingredients to decrypt the problem structure when the objective functions themselves are unknown and need to be investigated in order to reveal, in the Darwinian evolution example of section \ref{sec:evolution}, \emph{how nature works}.


\section{Motivation: Darwinian evolution}
\label{sec:evolution}
The guiding principle in the Darwinian theory of evolution is often resumed as the ``survival of the fittest''. This is often viewed as an optimization principle, suggesting the existence of a hidden fitness function which the living beings are aimed to maximize. The problem comes when the function itself becomes the focus of the interest, because its definition resists to be ultimately discovered. J
Indeed, sooner or later one is faced with involutive statements like ``the fitness is the function which is maximized by the surviving individuals''. As an extremization of this paradox, \cite{Thurner:2010hk} proposes a model of evolution reproducing many phenomena observable in nature without having resort to the introduction of a fitness function. 

Indeed, the large variety of species and even of individuals in the same species, raises serious arguments against the definition of a unique and immutable optimization process followed by evolution. As we have already observed, this multiplicity of solutions is not a problem in multiobjective optimization, but it is rather one of its fingerprints.   
Recently \cite{Shoval:2012ke,Noor:2012hg,Schuetz:2012io}, it has been proposed that evolution could be driven by multiple objectives. More precisely, one could think that an individual, or a species, is characterized by a vector of traits $v=(v_1,\dots,v_n)$, i.e., quantitative measures like beak length, body size, area proportions of molar teeth, resuming its \emph{phenotype}. The space of traits to which $v$ belongs is called the \emph{morphospace}. Then one assumes that the fitness function $F$ is a monotonically increasing function of a number of performance functions  $f_1,\dots,f_m$, $f_i(v)$ representing how well a specific task $i$ is performed by a organism with phenotype $v$. When the performance function $f_i$ has a single optimum $v_i$, this is referred to as the \emph{archetype} for $f_i$. 
 It is reasonable that the comprehensive fitness function $F$ depends monotonically on the performances $f_i$, however, the different contributions of these performances can vary over time or geographical space or because of interaction with other species. Therefore all phenotypes belonging to the Pareto set of the MO problem defined by $f_1,\dots,f_m$ have the chance to be selected by evolution and be observed in wildlife. 

\section{A hierarchical decomposition of the local Pareto set in the general convex case}

\begin{proposition}[Hierarchical decomposition in the convex case]
 	Let $W\subseteq\R^n$ open and convex. Let $f_i:W\To \R$, $i=1,\dots,m$, $m\leq n$, be smooth and convex functions. Then the Pareto set is a \emph{curved $m-1$ simplex}, i.e., is diffeomorphic to an $m-1$ dimensional simplex, i.e., the convex hull of a set of $m$ points in general position in $\R^n$. Each one of the vertices of the curved $m-1$ simplex coincide with one of the optima of the $m$ functions taken separately. Every $k-1$ facet of the curved simplex corresponds to the Pareto optimal set of the MO problem defined by a subset of $k$ functions in $\set{f_1,\dots,f_m}$.
\end{proposition}

\begin{remark}
 	In the case that all the functions are spherically symmetric the Pareto set is exactly an $m-1$ simplex, i.e., the facets are flat.
\end{remark}

\begin{proof}[Hierarchical decomposition in the convex case] This proposition is presented and discussed graphically in \cite{Shoval:2012ke}.  An alternative proof follows straightforwardly from the characterization of the Pareto set in terms of first and second order derivatives by Smale \cite{Smale:1973km} (see also \cite{Lovison:2013uq,Lovison:2011uq,Smale:1975oy,Wan:1977yi,Wan:1975lt}).\phantom{\qed}
\end{proof}

\begin{figure}[htbp]
\begin{center}
\begin{center}
  \begin{tabular}{@{} ccc @{}}
    \hline
    \includegraphics[width=35mm]{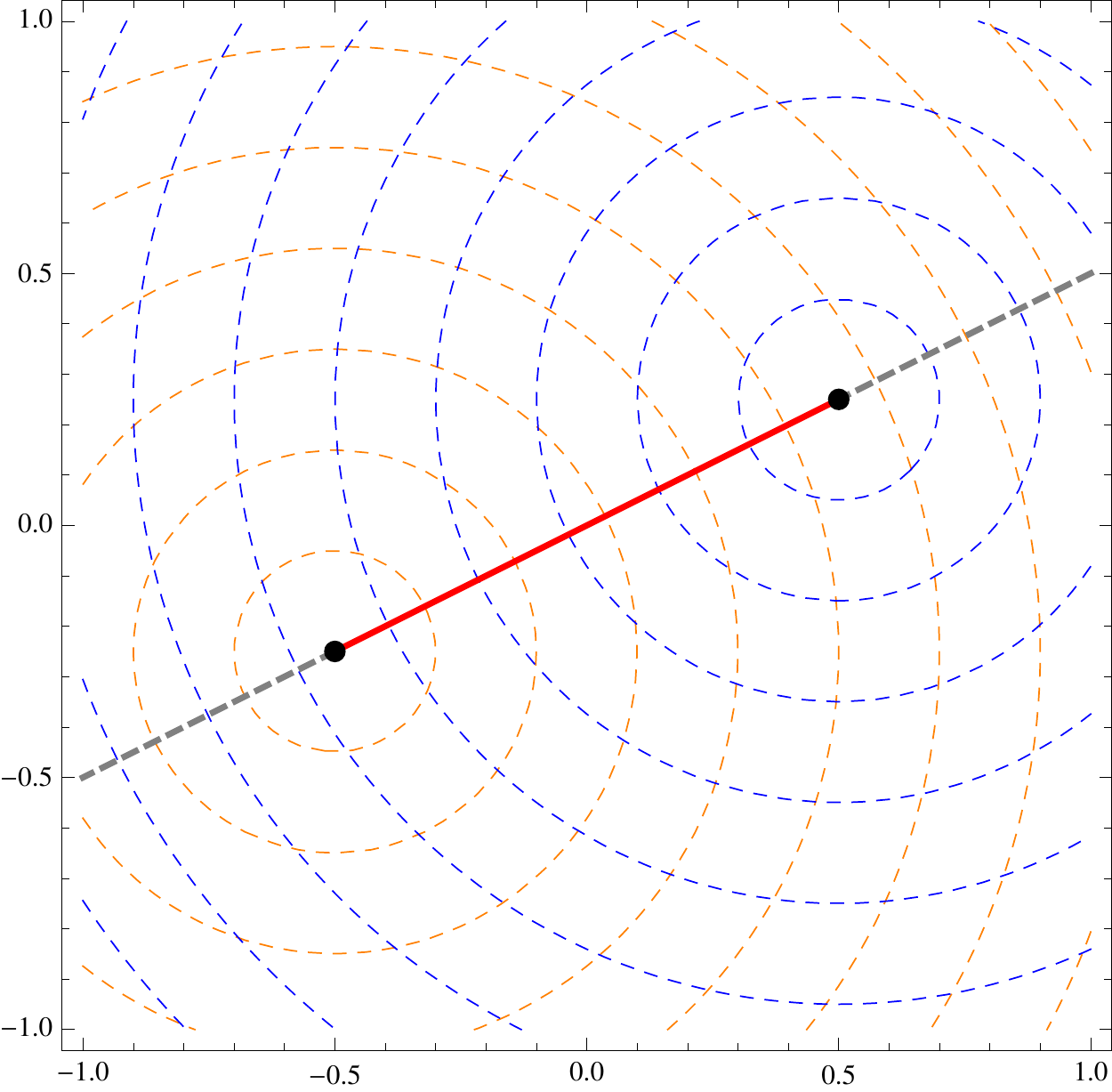} & \includegraphics[width=35mm]{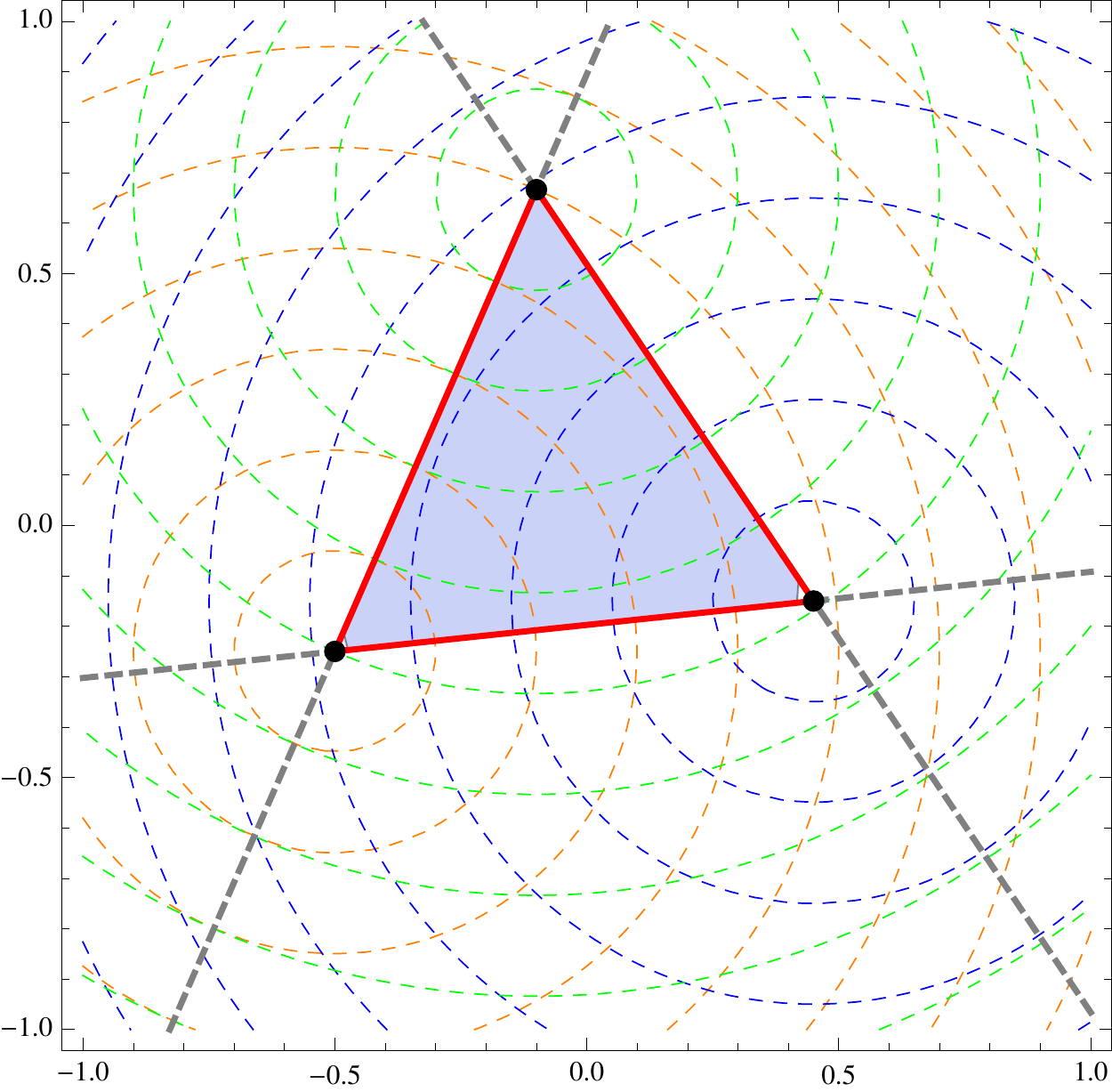} & \includegraphics[width=35mm]{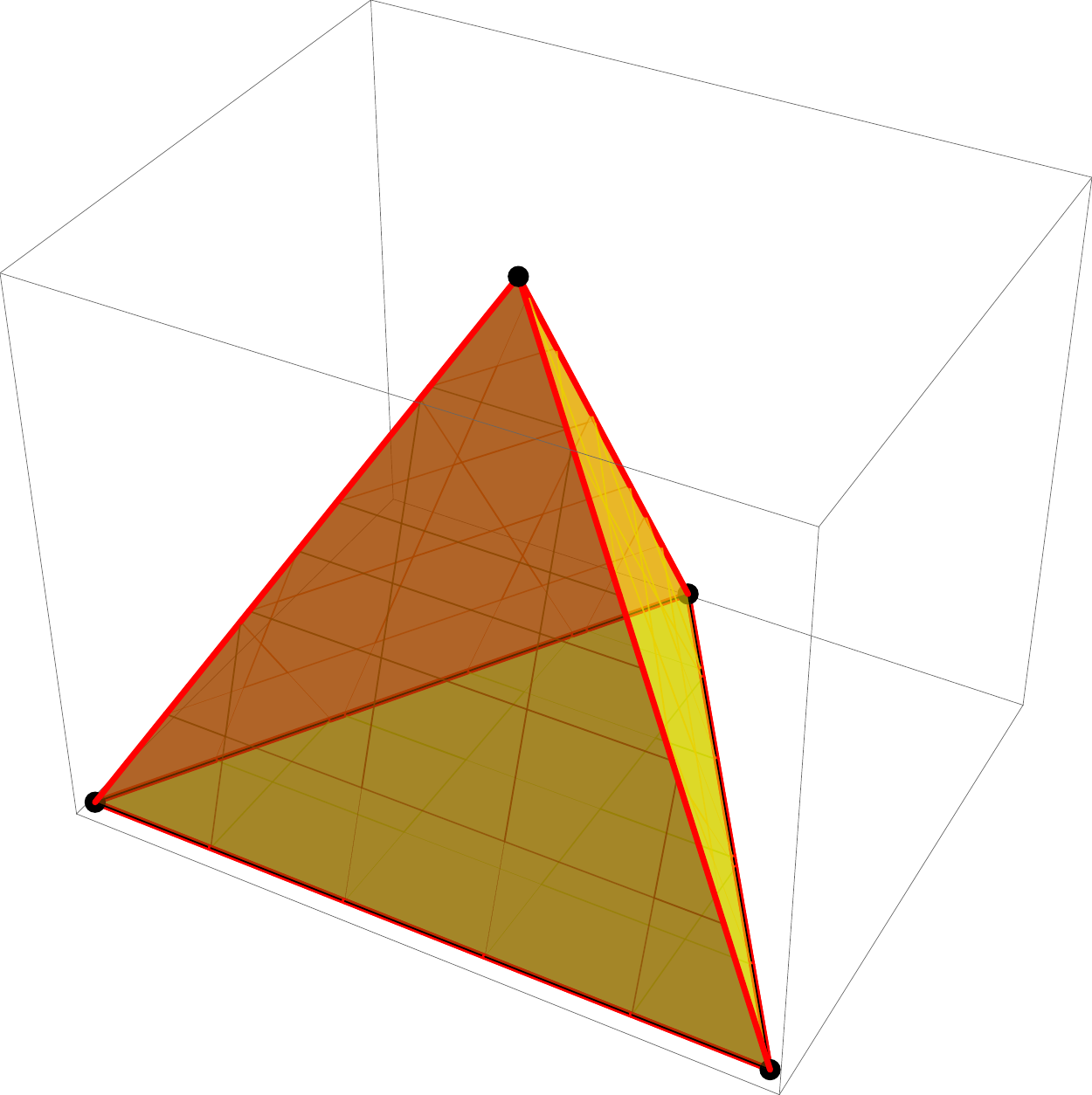} \\ 
    \hline
    (a) & (b) & (c) \\ 
    \hline
  \end{tabular}
\end{center}
\caption{Pareto sets in the convex case with spherical symmetry. (a) Two functions: the singular set (gray lines) is the locus where the gradients of $f_1$ and $f_2$ are parallel, i.e., where the level surfaces are tangent. The Pareto set (red lines) is the set where the gradients are parallel and oriented in opposite directions. The extrema are the optima of the two functions taken separately. (b) Three functions. The borders of the triangle are the Pareto optimal set for  the 2-objectives subproblems $\set{f_1,f_2}$, $\set{f_2,f_3}$ and $\set{f_3,f_1}$ while the corners are the optima of the three functions taken one by one. (c) Four functions. The facets of the tetrahedron are the optima of the 3-objectives subproblems. Analogous considerations hold for   lower dimensional elements of the skeleton.}
\label{fig:spherical_symmetry}
\end{center}
\end{figure}

\begin{figure}[htbp]
\begin{center}
\begin{center}
  \begin{tabular}{@{} ccc @{}}
    \hline
    \includegraphics[width=35mm]{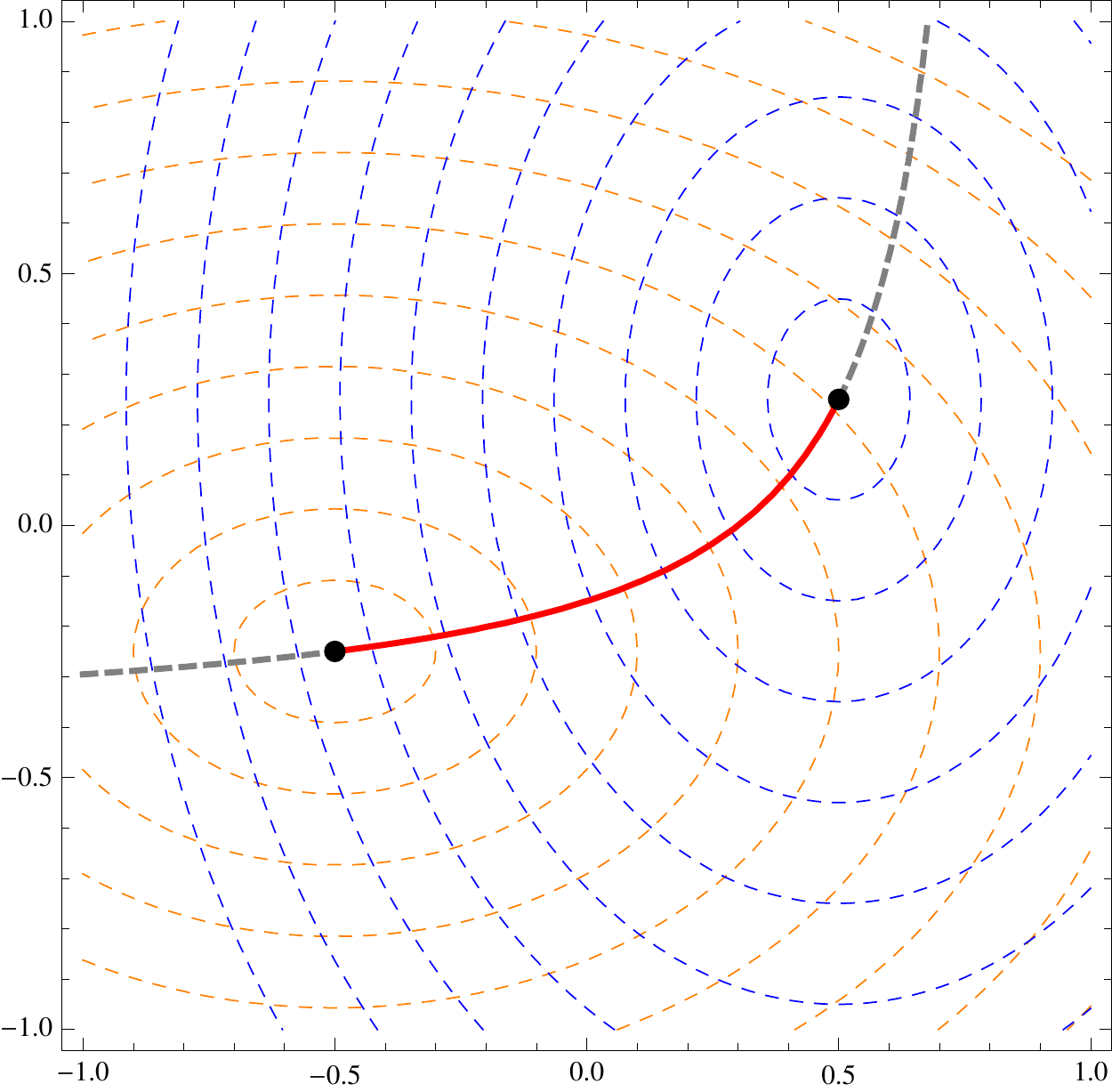} & \includegraphics[width=35mm]{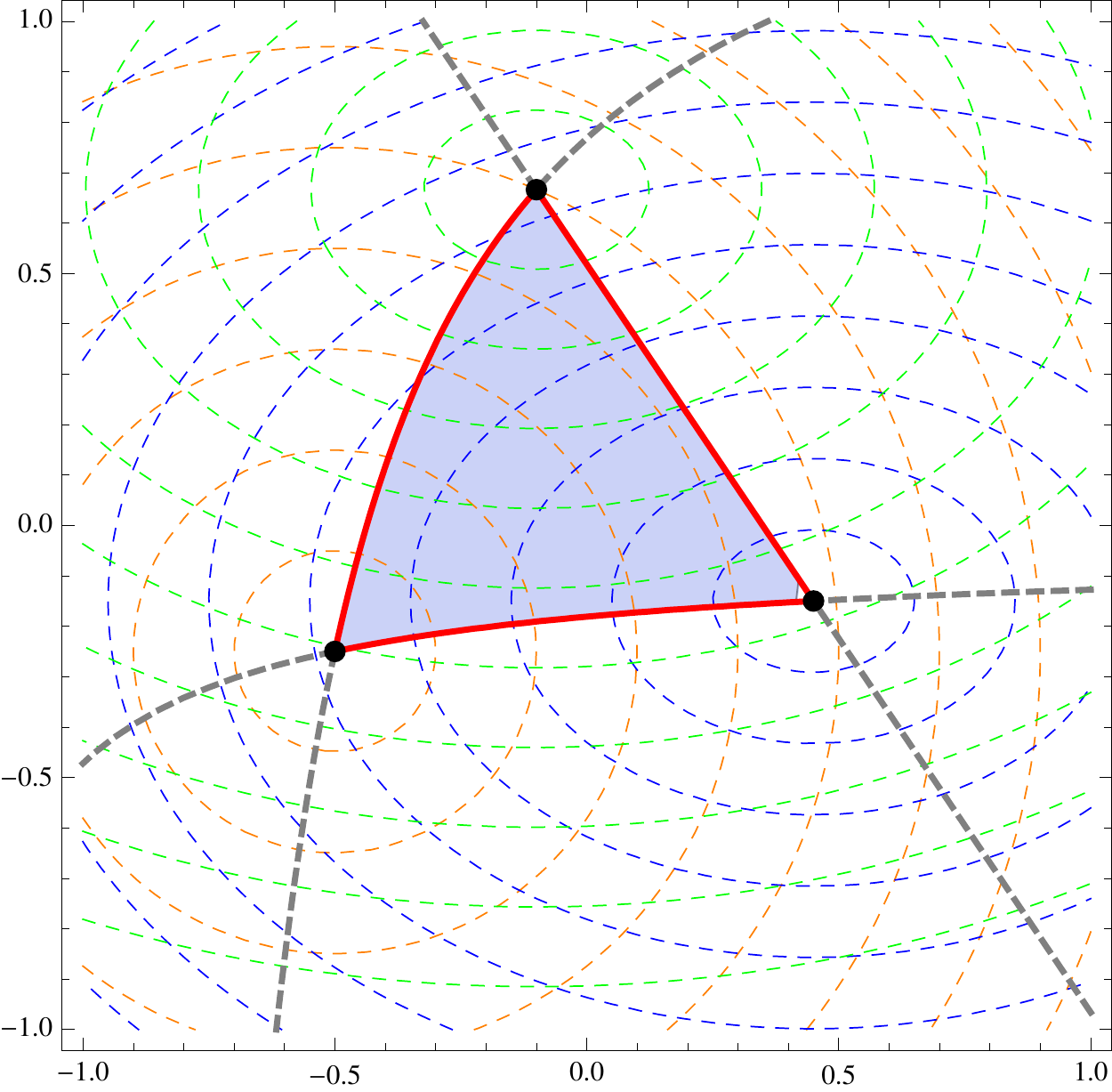} & \includegraphics[width=35mm]{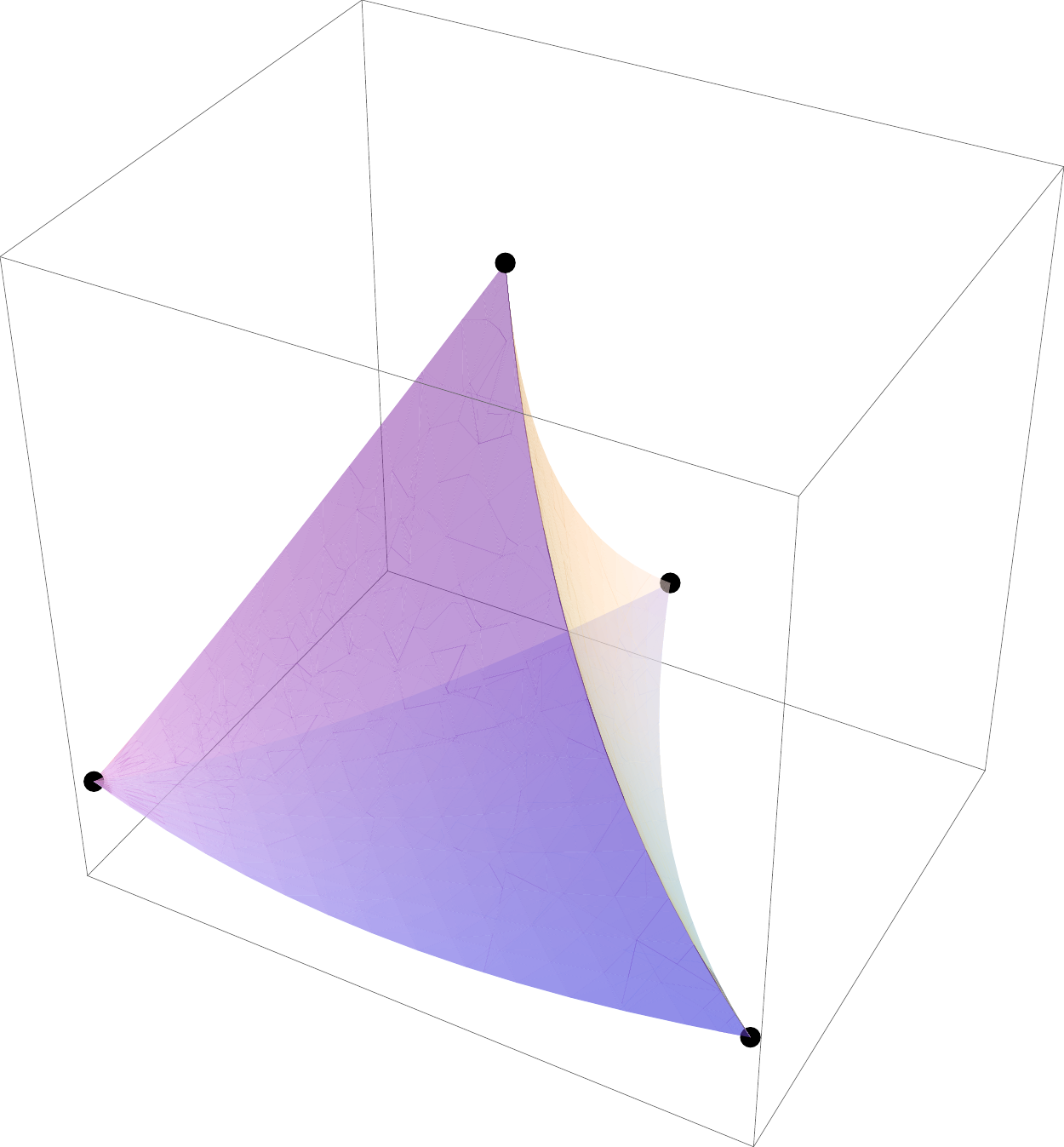} \\ 
    \hline
    (a) & (b) & (c) \\ 
    \hline
  \end{tabular}
\end{center}
\caption{Pareto sets in the general convex case. (a) Two functions. (b) Three functions. (c) Four functions. The curved $(k-1)$ skeleton of such curvilinear simplexes are related to the $k$-objectives subproblems, analogously to what happens in the spherically symmetric case.}
\label{fig:convex_case}
\end{center}
\end{figure}

\subsection{Darwinian evolution with multiobjective fitness}

In \cite{Shoval:2012ke}, the authors explore the variety in morphospace of a number of family of species, and by recognizing shapes assimilable to simplexes in the set of wild species.
See for instance the beak sizes and shapes for the Darwin's ground finches, leaf-cutter ants and microchiroptera  (Figure \ref{fig:triangles_in_evo}).
\begin{figure}[htbp]
\begin{center}
	\includegraphics[width=105mm]{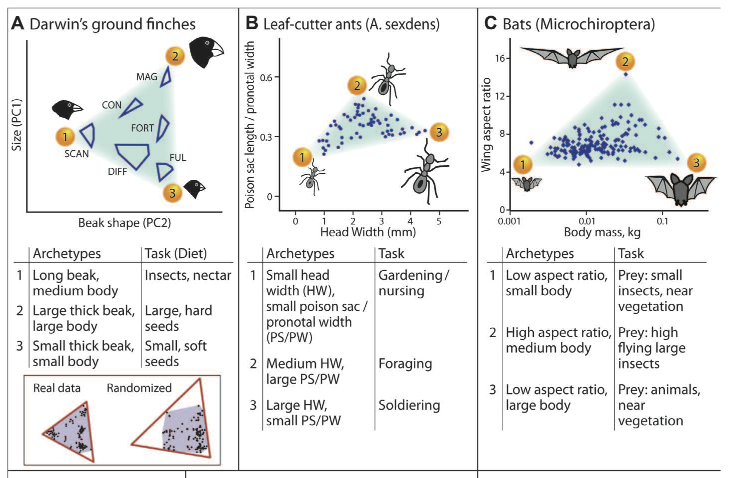}
\caption{Triangle-shaped sets of wild species observed in morphospace. From \cite{Shoval01062012}. Reprinted with permission from AAAS.}
\label{fig:triangles_in_evo}
\end{center}
\end{figure}
Even at the level of convex functions, the idea of the hierarchical decomposition of the Pareto set offers multiple advantages at least at two different levels. First of all, at epistemic level, the inspection of the Pareto set allows to detect the archetypes and therefore to decrypt the unknown objective functions. 
If this is the underlying principle of selection, the detection of the Pareto set and of the archetypes allows for a decryption of the structure of the multiobjective fitness function, revealing eventually the  
workings of the evolution mechanism, avoiding tautological loops.
Secondly, at numerical level, we have that solving the ordinary single objective functions 
one can detect the archetypes and then can obtain a zero-th order approximation of the Pareto set consisting simply in tracing the convex hull of the set of archetypes.
The approximation of the Pareto set with a simplex is exact when the objectives have rotational symmetry (e.g., quadratic forms with a unique eigenvalue). The simplex start to exhibit curvature as the objective functions depart from the spherical symmetry.
It is however worth noting that in the convex case we could obtain an essentially faithful approximation of the whole Pareto set simply solving $m$ single objective optimization problems and building their convex hull.

\section{Stratification of the Pareto set and sufficient regularity of objective functions}
In the convex case the Pareto set is diffeomorphic to an $(m-1)$ dimensional simplex, 
in the non convex case the situation can be much more complicated, analogously to what happens for the single objective case, where because of non convexity  multiextremality and non maximal critical points arise.

Nevertheless, in non pathological situations, the set of local Pareto optima appears still intelligible. Its geometrical structure can be analyzed and classifiable as an $(m-1)$ dimensional \emph{stratified set}. Heuristically, a stratified set is a \emph{manifold with borders and corners}, i.e.,  an $(m-1)$ manifold whose frontier is composed by manifolds of lower dimension with ``nice'' intersections. 

This section is rather technical and is  dedicated to a precise specification of what we mean by  
``nice intersection'', ``stratification'' and ``non pathological'', where we revisit in a unitary form the results in \cite{Melo:1976sj,Melo:1976xd,Wan:1977yi,Wan:1978jy}. 
Because of lacking of space we cannot give a self contained introduction on the singularities of differentiable mappings. The reader is referred to 
\cite{Arnold:1968fx,Golubitsky:1973uj}. 

There is also another allied concept which is necessary during the discussion, the \emph{Pareto criticality} \cite{Smale:1973km}. 
\begin{definition}[Pareto criticality]
Let $f:W \To R^m$ be smooth. Let $Pos$ be the positive orthant of $\R^m$, i.e., $Pos=\set{(y_1,\dots,y_m)\taleche y_j>0}$.
A point $p$ is said \emph{Pareto critical} if $\im Df(p) \cap Pos = \emptyset$.
\end{definition}
As it happens with ordinary scalar functions, criticality becomes important as we relax the assumption of convexity, because not only maxima  have zero derivative but also critical points of different nature arise, as minima or saddles. The condition mentioned is still important because it is a necessary condition for optimality and is a condition involving only first order derivatives.\footnote{Further discussion on Pareto criticality can be found in \cite{Lucchetti:2004rt,Miglierina:2008gf}.}

\subsection{Semialgebraic sets}
\begin{definition}
$C \subset \R^m$ is said to be \emph{semialgebraic} if and only if it can be defined as a finite union of sets $K_j$, each being defined by a finite set of polynomial equations and inequalities. In other words: 
\[
C= \bigcup \limits_{j=1}^n K_j \qquad K_j= \bigl \{ f_{_{i,j}}=0, \; g_{_{h,j}}>0, \quad i=1,\ldots,l, \quad h=1,\ldots,k \bigr \}
\] 
(where $f_{_{i,j}}$ and $g_{_{h,j}}$ are polynomial functions.)
\end{definition}

The closure of a semialgebraic set $C$, denoted by $\mathrm{Cl}(C)$, is defined by converting all the strict inequalities in the definition of $C$ to weak inequalities. $\mathrm{Cl}(C)$ is still semialgebraic.

\begin{theorem}\label{th:semialg}
A subset of $\R^m$ is semialgebraic if and only if it is in the smallest family of subsets of $\R^l$ closed under finite union, finite intersection and complements and which contains sets of the form $\bigl \{ g > 0 \bigr \}$, where $g$ is a polynomial in $m$ variables with real coefficients.
\end{theorem}
\begin{proof}
See Seidenberg \cite{Seidenberg1954}.
\end{proof}

\begin{theorem}[Tarski- Seidenberg]\label{Tar-Sei}
The image of a semialgebraic subset in $\R^m$ under a polynomial mapping from $\R^m$ into $\R^n$ is a semialgebraic subset in $\R^n$.
\end{theorem}
\begin{proof}
See Seidenberg \cite{Seidenberg1954}.
\end{proof}

\subsection{$r$th order conditions}

The $r$-\emph{jet} of a smooth ($C^\infty$) function $f:\R\To\R$ with source in $x_0$ and target in $f(x_0)$, is the equivalence class of the Taylor expansion of the function $f$ in $x_0$, arrested at the degree $r$. The $r$-jet of $f$ at $x_0$ is denoted by $j^r f(x_0)$ and can be identified with the Taylor expansion of $f$:
\begin{equation}
	j^r f(x_0)(x) = f(x_0) + f'(x_0)(x-x_0) + \dots+ \frac{f^{(r)}(x_0)}{r!}(x-x_0)^r.
\end{equation}
This definition can be generalized naturally to the mappings between euclidean spaces. 
Let $J^r(n,m)$ be the space of $r$-jets from $\R^n$ to $\R^m$, with source and target in $0$:
\[
J^r(n,m):=  \bigl \{ j^{r}f=j^rf(0) \big | f : \R^n \rightarrow \R^m, \, f(0)=0, \, \text{of class} \; C^r \bigr \}.
\]
This space is called Jet Bundle and we have that $J^r(n,m)$ is a vector space and that there exists a linear isomorphism given by
\begin{align*}
J^r (n,m) & \stackrel{\cong }{\To} B^r_{n,m} \\
j^rf &\longmapsto (j^r f_1,...,j^r f_m)
\end{align*}
where $B^r_{n,m}$ is the vector space of polynomial mappings from $\R^n$ to $\R^m$ with deg $\leq r$.
\begin{definition}
A subset $C \subset J^r(n,m)$ is called a \emph{$r$th-order condition} provided the set $C$ is invariant under the group of $C^r$ diffeomorphisms around $0 \in \R^n$. Moreover, if the set $C$ is semialgebraic we say that it is a $r$th-order algebraic condition.
\end{definition}

Now we generalize the definition of Jet Bundle to smooth manifolds, for further details and proofs the reader is referred to \cite{Golubitsky:1973uj}.

\begin{definition}
Let $X,Y$ be smooth manifolds, $p \in X$ and $r>1$. Let $f,g: X \rightarrow Y$ be $C^\mu$ ($\mu > r$) mappings such that $f(p)=g(p)=q$.
\begin{enumerate}
\item $f \sim_1 g$ at $p$ $\overset{def}{\Leftrightarrow}$ $df(p)=dg(p)$ as linear mappings $T_pX \rightarrow T_qY$.
\item $f \sim_r g$ at $p$ $\overset{def}{\Leftrightarrow}$ $df(p) \sim _{r-1} dg(p)$ at every point of $T_pX$.
\item Let $J^r(X,Y)_{p,q}$ be the set of the equivalence classes under the relation ``$\sim_r$ at $p$" of mappings $f:X \rightarrow Y$, with $f(p)=q$.
\item Set $J^r(X,Y) = \bigsqcup_{(p,q) \in X \times Y} J^r(X,Y)_{p,q}$. An element $z \in J^r(X,Y)$ is called \emph{$r$-jet}.
\item Let $z$ be a $r$-jet. Then there exist $p \in X$ and $q\in Y$ such that $z \in J^r(X,Y)_{p,q}$ : $p$ is called \emph{source of $z$} and $q$ is called \emph{target of $z$}.
\end{enumerate}
\end{definition}

Note that given $f:X\rightarrow Y$ is defined canonically the mapping
\[
\begin{split}
j^rf: \, &X \rightarrow J^r(X,Y)\\
&p \mapsto j^rf(p)\text{ = the equivalence class of $f$ in $J^r(X,Y)_{p,f(p)}$}
\end{split}
\]

\begin{lemma}
Let $U \subset \R^n$ be an open set, $p \in U$. Let $f,g: \, U \rightarrow \R^m$ be smooth mappings. Then
\[
f \sim_r g \, \Leftrightarrow \, \frac{\partial^{|\alpha|}f_i}{\partial x^\alpha}(p)= \frac{\partial^{|\alpha|}g_i}{\partial x^\alpha}(p) \quad \forall \alpha \in \N^n_0,\, |\alpha|\leq r, \, \forall i=1,...,n
\]
\end{lemma}

\begin{corollary}\label{cor:taylor}
$f,g : U \rightarrow \R^m$ are such that $f \sim_r g$ at $p$ $\Leftrightarrow$ the Taylor expansions of $f$ and $g$ up to order $r$ are identical at $p$.
\end{corollary}

\begin{proposition}
Let $X,Y$ be smooth manifolds of dimension $n$ and $m$ respectively. Then $J^r(X,Y)$ is a fiber bundle over $X\times Y$ with fiber $B^r_{n,m}$.
\end{proposition}

Let $W$ be a $C^\mu$ ($\mu>r$) compact manifold without boundary of dimension $n$. We have that $J^r(W,\R^m)$ is a vector bundle with fiber $B^r_{n,m} \cong J^r(n,m)$. \\
If $A \subset J^r(n,m)$ is a $r$th-order algebraic condition, then we can define a \emph{semialgebraic subbundle} of the vector bundle $J^r(W,\R^m)$: is the subbundle with fiber $A$ and we will denote it with $A(W,\R^m)$.
\begin{definition}
A $C^{\mu}$ map $f:W\rightarrow \R^m$ is said to satisfy a $r$th-order algebraic condition $A$ ($\mu > r$) if $j^rf(W) \subset A(W,\R^m)$. 
\end{definition}
An $r$-th algebraic condition is generic if is satisfied by a residual subset of $C^{\mu}(W,\R^m)$.

\subsection{Stratifications}
Roughly speaking, a stratification of a set $A$ is a partition of $A$ into manifolds fitting togheter regularly.
Now we give a precise mathematical definition for two submanifold of $R^n$ to have ``nice'' intersections, it is based on the idea that if we have two submanifold $X$ and $Y$ with $Y \subset Cl(X)$, then irregularity of their intersection at a point $y \in Y$ corresponds to the existence on nonunique limit of tangent planes $T_xX$ as $x$ approches to $y$. For further details the reader is referred to Mather \cite{Mather:1973fk, Mather:2012gp}. \\
Let $X,Y$ be $C^\mu$ submanifolds of $\R^n$.
\begin{definition}[Whitney's condition $b$]
 We say that the pair $(X,Y)$ \emph{satisfies Whitney's condition $b$ at $y$} if for each sequence $(x_k)_{k\in\N}\subset X$ and each sequence $(y_k)_{k\in\N}\subset Y$, both converging to $y\in Y$, such that the sequence of secant lines $(\overline{x_ky_k})_{k\in\N}$ converges to some line $l\subset \R^n$ (in the projective space $\mathbb{P}^{n-1}$) and the sequence $(T_{x_k}X)_{k\in\N}$ converges to some plane $\tau\subset\R^n$, we have that $l\subset \tau$. 
\end{definition}

Let $W$ be a $C^\mu$ manifold without boundary and $A$ be a subset of $W$. 
\begin{definition}[Stratifications] A \emph{stratification} $\mathcal{S}$ of $A$ is a cover of $A$ by pairwise disjoint submanifolds of $W$ which lie in $A$. The elements of $\mathcal{S}$ are called \emph{strata}.
$\mathcal{S}$ is said \emph{locally finite} if each point in $A$ has a neighborhood which meets only a finite number of strata. $\mathcal{S}$ satisfies the \emph{condition of the frontier} if for each stratum $\sigma$ of $\mathcal{S}$, $\partial \sigma \cap A$ is a union of strata in $\mathcal{S}$. 
\end{definition}

\begin{definition}[Whitney stratifications]
A stratification $\mathcal{S}$ of $A$ is said a \emph{Whitney stratification} if it is locally finite, it satisfies the condition of the frontier and if $(\sigma,\rho)$ satisfies Whitney's condition $b$ for any pair $(\sigma,\rho)$ of strata of $\mathcal{S}$.
\end{definition}

\begin{proposition}\label{prop:back_stratification}
Let $X,Y$ be $C^r$ manifolds and $A \subset Y$ closed set with a Whitney stratification $\mathcal{S}$. Let $f:X \rightarrow Y$ be a $C^{\mu}$ map transversal to $\mathcal{S}$ (with $\mu > r+1$). Define
\[
f^*\mathscr{S}= \bigr \{ \text{Connected components of } f^{-1}(\sigma) \big | \sigma \in \mathscr{S} \bigl \}
\]
Then $f^*\mathscr{S}$ is a Whitney stratification of $f^{-1}(A)$.
\end{proposition}
\begin{proof}
See Mather \cite{Mather:2012gp}.\phantom{\qed}
\end{proof}
\subsection{Stratifications of semialgebraic sets}
Consider a semialgebraic set $C\subset J^p(n,m)$, as we said above :
\[
C= \bigcup \limits_{j=1}^n K_j \qquad K_j= \bigl \{ f_{_{i,j}}=0, \; g_{_{h,j}}>0, \quad i=1,\ldots,l, \quad h=1,\ldots,k \bigr \}
\] 
(where $f_{_{i,j}}$ and $g_{_{h,j}}$ are polynomial functions.)\\
Every semialgebraic set admits a Whitney's stratification (see Mather \cite{Mather:1973fk}), whose strata are submanifolds determined by the polynomial inequalities defining the set. We denote with $\mathcal{S}_0$ this Whitney's stratification. \\
The closure $\mathrm{Cl}(C)$ is still semialgebraic, is defined by converting all the strict inequalities in the definition of $C$ in weak inequalities. Thus also $\mathrm{Cl}(C)$ is stratified: the strata are defined by the polynomial inequalities, we call this last stratification $\widehat{\mathcal{S}_0}$.\\
Note that $\mathcal{S}_0 \subset \widehat{\mathcal{S}_0}$ but this second one is not merely the topological closure of the first: $\widehat{\mathcal{S}_0}$ contains the strata of $\mathcal{S}_0$ and other strata in addition, which appear by converting each strict inequality to a weak inequality.
Hence $C$ is a \emph{Whitney's substratified} subset of $\mathrm{Cl}(C)$, i.e., is union of strata of $\mathrm{Cl}(C)$. \\
Now assume that $C\subset J^r(n,m)$ is a $r$-th order algebraic condition and let let $W$ be a $C^\mu$ compact manifold without boundary of dimension $n$. Then we can define $C(W,\R^m)$, semialgebraic subbundle of $J^r(W,\R^m)$ with fiber $C$.\\
Note that $C(W,\R^m)$ is not a smooth subbundle, since $C$ is not a smooth submanifold, but it is a stratified set: $\mathcal{S}_0$ induces a Whitney stratification $\mathcal{S}$ on $C(W,\R^m)$.\\
Moreover denote with $\widehat{C}(W,\R^m)$ the semialgebraic subbundle with fiber $\mathrm{Cl}(C)$ (note that this is not the topological closure of $C(W,\R^m)$). Then $\widehat{\mathcal{S}_0}$ induces a Whintey's stratification $\widehat{\mathcal{S}}$ on $\widehat{C}(W,\R^m)$. \\
Finally we have that $\mathcal{S} \subset \widehat{\mathcal{S}}$, i.e. $C(W,\R^m)$ is a substratified subset of $\widehat{C}(W,\R^m)$. 

\subsection{Algebraic conditions for Pareto optimality}
Throughout the following sections, let $W$ be a $C^\mu$ compact manifold without boundary of dimension $n$.

In section we show an algebraic necessary and sufficient condition for Pareto optimality. The main reference is Wan \cite{Wan:1977yi}.
\begin{definition}
An $r$-jet $z \in J^r(n,m)$ is called \emph{v-sufficient} if for any two realizations $f$,$g$ of $z$, $f^{-1}(0)$ is homeomorphic to $g^{-1}(0)$ near $0$.
(We call $f$ a realization of $j^rf$).
\end{definition}
Recall that an $r$-jet $z \in J^r(n,m)$ can be represented as a polynomial mapping in the variables $x=(x_1,\ldots,x_n)$ of degree less or equal to $r$. We can give the following characterization theorem of v-sufficiency of $r$-jets.
\begin{theorem}\label{th:kuo}
For a given jet $z=(z_1,\ldots,z_m) \in J^r(n,m)$, the following conditions are equivalent:
\begin{itemize}
\item $z$ is $v$-sufficient in $C^r$.
\item There exists $\epsilon>0$ such that $d(\nabla z_1(x), \ldots, \nabla z_m(x))\geq |x|^{r-1}$ for $|z(x)|\leq\epsilon|x|^r$ and $|x|<\epsilon$. 
\item Given any $C^r$ realization $f$ of $z$, $\nabla f_1(x),\ldots,\nabla f_m(x)$ are linearly independent on $f^{-1}(0)\setminus\{0\}$, near $0$.
\end{itemize}
(we set $d(v_1,\ldots v_m)$=min$\{h_1,\ldots,h_m\}$, where $h_i$=distance from $v_i$ to the subspace spanned by the vectors $v_j$, $j\neq i$.)
\end{theorem}
\begin{proof}
See Kuo \cite{Kuo:1972fk}.
\end{proof}

Unfortunately, establishing the validity of Kuo conditions, as they are expressed above, is non trivial. The first problem is that the function $d(v_1,\ldots,v_m)$ is  given by an impractical formula. This causes difficulties in the effective evaluation of $d(v_1,\ldots,v_m)$. The second problem is that $d(\nabla z_1(x), \ldots, \nabla z_m(x))$ has to satisfy an inequality in a horn neighbourhood  of of $z^{-1}(0)$ and not in an ordinary neighbourhood of the origin. 

On the other hand, there exists an equivalent formulation \cite{Koz:2010prk} of the v--sufficiency for jets which consists in the computation of the local Lojasiewicz exponent for an associated polynomial function. We discuss this issue in more details in the appendix.

Even if not easily verified, Kuo conditions of Theorem \ref{th:kuo} are naturally semialgebraic and hence using Tarski-Seidenberg theorem one can obtain the following proposition: 
\begin{proposition}\label{th:vsuff_semialg}
The set of all $r$-jets which are not v-sufficient forms a semialgebraic subset of codimension greater or equal to $|n-m|+r$ in $J^r(n,m)$.
\end{proposition}
\begin{proof}
 	See Wan \cite{Wan:1977yi}.
\end{proof}
Given a $r$-jet $z=(z_1,\ldots,z_m) \in J^r(n,m)$ and a non-empty subset $S=(s_1,\ldots,s_a) \subseteq (1,\ldots,m)$, set $z_S=(z_{s_1},\ldots, z_{s_a})\in J^r(n,a)$.
\begin{definition}
Let 
\[
V_r=\bigl\{z \in J^r(n,m)\,\big | \, z_S \, \text{is v-sufficient for any non-empty subset }S \subseteq (1,\ldots,m) \bigr \}.
\]
\end{definition}

\begin{lemma} \label{lm:vr_semialg}
$V_r$ is a semialgebraic set and $codim \, V_r^c \geq max\{n-m,0\} +r$.
\end{lemma}
\begin{proof}
Let 
\[
B_S= \bigl \{ z \in J^r(n,m) \, \big | \, z_s\text{ is not $v$-sufficient} \,\bigr\}, \quad S \subseteq (1,\ldots,m).
\]
By Proposition \ref{th:vsuff_semialg} $B_S$ it's a semialgebraic set, moreover we have
\[
V_r= \bigcap_{S\subseteq (1,\ldots,m)} B^{c}_S= \Bigl( \bigcup_{S\subseteq (1,\ldots,m)} B_S\Bigr)^c.
\]
Hence $V_r$ is a semialgebraic set from Theorem \ref{th:semialg}. \\
Finally note that
\[
dim \, V_r^c \leq \sum_{S}{dim \, B_S} \leq dim \, J^r(n,m) -\sum_{S}{|n-a|+r} ,
\]
where $S=(s_1,\ldots,s_a) \subseteq (1,\cdots,m)$. Hence
\[
codim \, V_r^c \geq \sum_{S\subseteq(1,\ldots,m)}{|n-a|+r} \geq max\{n-m, 0\} +r,
\]
in fact
\[
\begin{split}
						&n\geq m \Rightarrow n \geq a \Rightarrow |n-a|=n-a \geq n-m. \\
						&n \leq m \Rightarrow |n-a| \geq 0.
\end{split}
\]
\phantom{\qed}
\end{proof}

\begin{lemma}\label{lm:vr_inv}
$V_r$ is a $r$th-order algebraic condition.
\end{lemma}
\begin{proof}
Consider in $\R^n$ a local diffeomorphism near $0$ $\phi:U \rightarrow V$, with $\phi(0)=0$. Let $z \in J^r(n,m)$ be a v-sufficient jet and let $f,g$ be two realizations of $z$.
Then $f\circ\phi$ and $g\circ\phi$ are two realization of the same $r$-jet $z* \in J^r(n,m)$ and $\phi^{-1}\circ f^{-1}(0)$ and $\phi^{-1}\circ f^{-1}(0)$ are homeomorphic.
Hence
\[
V_r=\bigl\{z \in J^r(n,m)\,\big | \, z_s \, \text{is v-sufficient, } \forall S \subseteq (1,\ldots,m) \bigr \}.
\]
is invariant under the group of local diffeomoprhisms near $0$.\\
Using also Lemma \ref{lm:vr_semialg} we conclude that $V_r$ is a $r$th-order algebraic condition. \phantom{\qed}
\end{proof}

\begin{proposition} \label{th:pareto_stab} 
Suppose $f_0,f_1\,:\,U\subset\R^n \rightarrow \R^m$ are two realizations of the same $r$-jet in $V_r \subset J^r(n,m)$. If $f_0$ has a local Pareto optimum at the origin then also $f_1$ has a local Pareto optimum at the origin.
\end{proposition}
For a proof of this Proposition see Wan \cite{Wan:1977yi}. Here we just want to outline that if one merely assume that $z\in J^r(n,m)$ is v-sufficient then the conclusion of Proposition \ref{th:pareto_stab} is not necessarily true. 

\begin{remark}
Consider
\[
\begin{split}
&z(x,y)=(-x^2,y) \in J^2(2,2) \\
&f_0(x,y)=(-x^2 - y^4,y) \\
&f_1(x,y)=(-x^2+y^4,y).
\end{split}
\]
Clearly $z=j^2f_0=j^2f_1$ is v-sufficient in $C^2$, since for any realization $f,g$ of $z$ we have that $f^{-1}(0,0)$ and $g^{-1}(0,0)$ are homeomorphic to $(0,0)$. \\
We note that in this case $(0,0)$ is a strict local Pareto optimum for $f_0$ but is not local Pareto optimum for $f_1$. \\
The problem here is that $z$ is not in $V_2$. Let $z_1(x,y)=-x^2$ be the first component of $z$ and let $g(x,y)=-x^2-y^4$, $h(x,y)=-x^2+y^4$ be two realization of $z_1$. \\
Then $g^{-1}(0)=(0,0)$ is not homeomorphic to $h^{-1}(0)=\Bigl \{(x,y)\in \R^2 \taleche y^2=x^2 \Bigr \}$, hence $z_1$ is not v-sufficient.
\end{remark}
Let $V_r(W,\R^m)$ be the semialgebraic subbundle of $J^r(W,\R^m)$ with fiber $V_r$. For convenience, we reformulate a proof in this setting of a result in \cite{Wan:1977yi}.
\begin{proposition}
Let $\mu > min(n,m)+1$ and set $q=min(n,m)+1$. Condition $V_q$ is generic in $C^{\mu}(W,\R^m)$.
\end{proposition}
\begin{proof}
Set 
\[
\begin{split}
M &=\bigl \{ f \in C^{\mu}(W,\R^m) \, \big | \, f \, \text{satisfies} \, V_q \bigr \} \\
&=\bigl \{f\in C^{\mu}(W,\R^m) \, \big| \, j^qf(W) \subset V_q(W,\R^m) \bigr \}\\
&=\bigl \{f\in C^{\mu}(W,\R^m) \, \big| \, j^qf(W) \cap V_q^c(W,\R^m)= \emptyset \bigr \}.
\end{split}
\]
By Lemma \ref{lm:vr_semialg} $codim\,V_q^c>n$ hence we have that
\[
M = \bigl \{f\in C^{\mu}(W,\R^m) \, \big| \, j^qf \, \tpitchfork \, V_q^c(W,\R^m) \bigr \}.
\]
Thus, applying Thom's Transversality, we can conclude that $W$ is a residual subset of $C^{\mu}(W,\R^m)$. \phantom{\qed}
\end{proof}

Let 
\[
\begin{split}
A=\bigl \{ & j^q f \taleche f:\R^n \To \R^m \text{ is a polynomial of deg $\leq q$}  \\ 
& \text{and has a local Pareto optimum at } 0 \bigr \}.
\end{split}
\]

Polynomial mappings from $\R^n$ to $\R^m$ , which have degree less than or equal to $q$ and vanish at zero, form a vector space $B^q_{n,m}$ and we can take as coordinates the coefficients of the polynomials so that $B^q_{n,m}$ becomes isomorphic to some $\R^l\cong\{(a_\iota) \, | \, a_\iota \in \R^m, \, 0<|\iota|\leq q \}$. \\
Given $(a_\iota)\in \R^l$, we set $f(a,x)= \sum_i{a_\iota x^\iota}$. Thus $f(a,\cdot): \R^n \rightarrow \R^m$ is the polynomial mapping with coefficients $(a_\iota)$.
Let $f(a,x)=(f_1(a,x), \ldots , f_m(a,x))$.
\begin{lemma}
The set
\[
\begin{split}
B=\bigl \{& a \in \R^l \big|f_1(a,y) \geq 0, \ldots , f_m(a,y) \geq 0 \\
& \Rightarrow \, f_1(a,y)=0,..., f_m(a,y)=0, \, \forall y\in \R^n, \,|y|^2<1 \bigr \}
\end{split}
\]
is a semialgebraic subset in $\R^l$.
\end{lemma}
\begin{proof}
By Theorem \ref{th:semialg} it suffices to prove that $B^c$ is semialgebraic. Let 
\[
\pi_1: \, \R^l \times \R^n \rightarrow \R^l, 
\]
and set
\[
\begin{split}
\overline{Y}=\bigl \{&(a,y) \in \R^l \times \R^n \, \big| \, |y|^2<1, \, f_1(a,y) \geq 0, \ldots , f_m(a,y) \geq 0, \\
										 &f_1(a,y)+ \ldots +f_m(a,y)=0 \bigr \}.
\end{split}
\]
Cleary $\overline{Y}$ is semialgebraic, $\pi_1(\overline{Y}) = B^c$ and $\pi_1$ is a polynomial mapping. Using Taski-Seidenberg we conclude that $B^c$(and hence $B$) is semialgebraic. \phantom{\qed}
\end{proof}
\begin{theorem}\label{th:pol_semialg}
Let 
\[
E= \bigl \{ a\in\R^l \, \big| \, f(a, \cdot) \, \text{has a local Pareto optimum at }0 \bigr \}.  
\]
Then E is a semialgebraic subset in $\R^l$.
\end{theorem}
\begin{proof}
Consider the polynomial mapping
\[
\rho: \R \times \R^l \rightarrow \R^l, \quad \rho(\epsilon, a)=(\epsilon^{k-|\iota|}a_\iota)_\iota
\]
Set
\[
\widetilde{E}=\bigl \{ (\epsilon,a) \in \R \times \R^l \, \big| \, \epsilon>0, \, a \in B \bigr \}.
\]

Clearly $\widetilde{E}$ is semialgebraic (since $B$ is semialgebraic). Moreover $\rho(\widetilde{E})=E$. \phantom{\qed}
\end{proof}
$A$ and $E$ are naturally isomorphic by means of the choice of the coefficients of the polynomials as coordinates, hence we have that $A$ is a semialgebraic subset of $J^q(n,m)$. Moreover, $A$ is clearly invariant under the group of local diffeomorphism of $\R^n$ around $0$. 
Define $C=V_q \cap A$. Then $C$ is a $q$-th algebraic condition and we can define $C(W,\R^m)$, the semialgebraic subbundle of $J^q(W,\R^m)$ with fiber $C$. 

With the following result Y.H.Wan \cite{Wan:1977yi} proves that for a residual subset of $C^\mu(W,\R^m)$ we have a necessary and sufficient condition for Pareto local optimality.
\begin{proposition}
Let $\mu > q$. Then a $C^{\mu}$ map $f: W \To \R^m$ satisfying condition $V_q$ has a local Pareto optimum at a point $x\in W$ if and only if $f$ satisfies condition $C$ at $x$, i.e., $j^qf(x) \in C(W,\R^m)$.
\end{proposition}
\begin{proof}
Recall that $q=min(n,m)+1$. \\
Let $C \subset J^q(n,m)$ be the $q$th-order algebraic condition defined above.
Set
\[
M=\bigl \{ f \in C^{\mu}(W,\R^m) \, \big | \, f \, \text{satisfies} \, V_q \bigr \}
\]
and let $f\in M$. \\
Assume f has a local Pareto optimum at $x \in W$. This implies $j^qf(x) \in C(W,\R^m)$.
On the other hand, assume $j^qf(x) \in C(W,\R^m)$. Then by Proposition \ref{th:pareto_stab} $f$ has a local Pareto optimum at x. \\
Hence we conclude that $C$ is a necessary and sufficient $q$th-order algebraic condition for $x$ to be a local Pareto optimum of $f$. \phantom{\qed}
\end{proof}
\subsection{Sufficient regularity for a function}

The $q$-th order condition $C \subset J^r(n,m)$ of the previous section is algebraic, hence by definition it is a semialgebraic subset of $J^r(n,m)$, thus there exists a Whitney stratification $\mathcal{S}_0$ of $C$. 
\begin{itemize}
\item $C(W,\R^m)$ is not a smooth subbundle, since $C$ is not a smooth submanifold, but it is a stratified set: $\mathcal{S}_0$ induces a Whitney stratification $\mathcal{S}$ on $C(W,\R^m)$. \\
\item $\mathrm{Cl}(C)$ is semialgebraic and admits a Whitney stratification $\overline{\mathcal{S}}_0$ s.t. $C$ is union of strata ($\mathcal{S}_0 \subset \overline{\mathcal{S}}_0$).\\
\item $\overline{C}(W,\R^m)$, the semialgebraic subbundle with fiber $\mathrm{Cl}(C)$, has a Whitney stratification $\overline{\mathcal{S}}$ s.t. $C(W,\R^m)$ is union of strata.
\end{itemize}
We define now an important class of functions.
\begin{definition}
Let $\mu > q +1$ ($q=min(n,m)+1$). Then a map $f:W \rightarrow \R^m$ of class $C^\mu$ is \emph{sufficiently regular} if
\begin{enumerate}
\item $f$ satisfies condition $V_q$.
\item $f \, \tpitchfork \, \sigma$, for all strata $\sigma \in \overline{\mathcal{S}}$
\end{enumerate} 
\end{definition}

\subsection{Main results}
\begin{theorem}
Let $\mu>min(n,m)+2$. Then the set of local Pareto optima for a \emph{sufficently regular} function $f:W \rightarrow \R^m$ admits a Whitney stratification.
\end{theorem}
\begin{proof}
Let $f:W\To \R^m$ be a sufficently regular function. By the previous proposition we have that
\[
\theta_{op}=j^qf^{-1}(C(W,\R^m))
\]
Moreover $\overline{C}(W,\R^m)$ admits a Whitney stratification $\overline{\mathcal{S}}$ s.t. $C(W,\R^m)$ is a substratified set, with stratification $\mathcal{S}$. Then by Proposition \ref{prop:back_stratification} $j^qf^{-1}(\overline{C}(W,\R^m))$ admits a Whitney stratification given by
\[
j^qf^*\overline{\mathcal{S}}= \set{\text{conncected components of } j^qf^{-1}(\sigma)\taleche \sigma \in \overline{\mathcal{S}}}
\]
Therefore, $\theta_{op} \subset j^qf^{-1}(\overline{C}(W,\R^m))$ admits a Whitney stratification: is union of strata of $j^qf^*\overline{\mathcal{S}}$. \phantom{\qed}
\end{proof}

\begin{proposition}
Let $\mu>min(n,m)+2$. The set of sufficiently regular functions is \emph{residual} in $C^\mu(M,\R^m)$ i.e., is
a countable intersection of dense sets in the Whitney $C^\mu$ topology. Furthermore, 
being $C^\mu(W,\R^m)$ endowed with the Whitney $C^\mu$ topology a Baire space, any residual set is also dense.
\end{proposition}
\begin{proof}
Let $G=\set{f:W \To \R^m \taleche f\text{ is sufficently regular} }$. Then $G=H \cap L$, where:
\begin{itemize}
\item $H=\set{f:W \To \R^m \taleche f\text{ satisfies condition }V_q}$ is residual.
\item $L=\set{f:W \To \R^m \taleche f \tpitchfork \sigma, \, \forall \sigma \in \overline{\mathcal{S}}}$ is residual by Thom's Transversality, since $\overline{\mathcal{S}}$ has finitely many strata. \phantom{\qed}
\end{itemize}
\end{proof}

Therefore, for a dense class of smooth functions in any reasonable topology, the Pareto set is a Whitney stratified set. This justifies the fact that one can expect that typically the Pareto set is composed by branches of $m-1$ dimensional manifolds with boundaries and corners, and if this is not the case, an arbitrarily small perturbation of the objective functions
will bring back the situation to the expected case.

Moreover, the stratification of the Pareto set can be described by polynomial inequalities, derivable from necessary and sufficient algebraic conditions for Pareto optimality. Hence it is important to find a more workable expression of Kuo--conditions. In the Appendix we propose a reformulation to approach the polynomial conditions for v-sufficiency in a more practical way.

\section{Conclusions}

Driven by the interpretation of the Darwinian evolution as a the ``survival of the fittest'' where the fitness function is multi objective, we have analyzed the hierarchical decomposition of the Pareto set in geometrical objects of dimensions $0,1,\dots,m-1$, where $m$ is the number of objective functions. 

For real world applications, such decomposition can give useful suggestions on the definition of the unknown function optimized by observed realizations, in the biological example, the wild species. 

For optimization problems, such decomposition may help to design effective optimization strategies, by composing in a consistent way the Pareto optimal sets of the smaller dimension subproblems, which are easier to study than the full problem with all the functions involved. 

It is our conviction that exploring and highlighting such hierarchical structures gives precious  
insights in MO problems, from the numerical point of view but also from the point of view of the decision maker, which will have in this way a more clear and structured idea of the whole range of possible solutions.

\section*{Appendix}
Here we want to give a reformulation of Kuo conditions in Theorem \ref{th:kuo}, for details see \cite{Koz:2010prk}. We will reduce the verification of Kuo conditions to the problem of the rate of growth of a polynomial about one of its roots, which is equivalent to calculation of the so called \emph{local Lojasiewicz exponent} of a polynomial. \\
Recall that according to the Lojasiewicz theorem \cite{Loj:1959,Loj:1991,Malgrange:1967} for any polynomial $p:\R^n \rightarrow \R^m$ with $p(0)=0$ there exist constants $C,k >0$ s.t.
\[
|p(x)| \geq C |x|^k
\]
in a neighborhood of the zero root. The last $k$ for which the above inequality holds is called the \emph{local Lojasiewicz exponent} for $p$. There is quite a number of publications devoted to evaluation of the Lojasiewicz exponent, see, e.g., \cite{abderrahmane:2005,Chadzynski:1998,Kurdyka:2005,Barroso:2003,Barroso:2005, Gwo:1999,Kollar:1999,Lenarcik:1997}. \\
Given a map $f:\R^n \to \R^m$ with $f(0)=0$ and an integer $p\geq 1$ define the following two functions in the variables $x \in \R^n$ and $y \in \R^m$ :
\[
\mathcal{R}_p(f;x,y)=|f(x)|^p|y|^p + |(df)^*(x) y|^p|x|^p
\]
and
\[
\mathcal{T}_p(f;x,y)=|f(x)|^p|y|^p + |(df)^*(x)y|^p|x|^p - |(df)^*(x)y \cdot x|^p.
\]

(Where we denote with $(df)^*(x)$ the conjugate of $df(x)$).  \\
Recall that an $r$-jet $z \in J^r(n,m)$ can be represented as a polynomial mapping in the variables $x=(x_1,\ldots,x_n)$ of degree less or equal to $r$. Then it holds the following theorem
\begin{theorem}[Kozyakin \cite{Koz:2010prk}]
Let $f \in C^r(\R^n,\R^m)$ with $n \geq m$. The jet $j^rf(0)=z$ is v-sufficient if and only if for evey $p \in \N$ there exist $q>0$, $\epsilon>0$ s.t.
\begin{equation}
\label{eq:koz}
\mathcal{X}(z;x,y) \geq q |x|^{pr}|y|^p \quad |x| < \epsilon, \forall y \in \R^m   
\end{equation}
Where $\mathcal{X}$ represents any one of the two polynomial functions $\mathcal{R}_p$ or $\mathcal{T}_p$.
\end{theorem}

\bibliographystyle{spmpsci}      
\bibliography{synthesis}   

\end{document}